\documentclass[review]{elsarticle}

\usepackage{lineno,hyperref}
\modulolinenumbers[5]

\journal{Information Sciences}

\usepackage[utf8]{inputenc} % allow utf-8 input
\usepackage[T1]{fontenc}    % use 8-bit T1 fonts
\usepackage{hyperref}       % hyperlinks
\usepackage{url}            % simple URL typesetting
\usepackage{booktabs}       % professional-quality tables
\usepackage{amsfonts}       % blackboard math symbols
\usepackage{nicefrac}       % compact symbols for 1/2, etc.
\usepackage{microtype}      % microtypography

\usepackage{amsfonts,amscd,amssymb}
\usepackage{amsthm,amsmath}
\usepackage{graphicx}

\newtheorem{Theorem}{Theorem}
\newtheorem{Lemma}{Lemma}

\newtheorem{Corollary}{Corollary}

\newtheorem{Assumption}{Assumption}

\newcommand{\mc}{\mathcal}
\newcommand{\mbb}{\mathbb}
\newcommand{\mb}{\mathbf}

\begin{document}

\begin{frontmatter}

\title{Posterior concentration and fast convergence rates for generalized Bayesian learning}

%\author{Lam Si Tung Ho}
%\address{Department of Mathematics and Statistics \\
%Dalhousie University, Halifax, Nova Scotia, Canada}


\author[add1]{Lam Si Tung Ho}
\author[add2,add3,add6]{Binh T. Nguyen}
\author[add4]{Vu Dinh}
\author[add5]{Duy Nguyen}


\address[add1]{Department of Mathematics and Statistics, Dalhousie University, Halifax, Nova Scotia, Canada}
\address[add2]{Department of Computer Science, Faculty of Mathematics and Computer Science, University of Science, Ho Chi Minh City, Vietnam}
\address[add6]{Vietnam National University, Ho Chi Minh City, Vietnam}
\address[add3]{AISIA Research Lab }
\address[add4]{Department of Mathematical Sciences, University of Delaware, USA}
\address[add5]{Department of Statistics, University of Wisconsin-Madison, USA}

%\aistatsauthor{Lam Si Tung Ho \And Binh T. Nguyen \And Vu Dinh}
%
%\aistatsaddress{ Dalhousie University\\ \texttt{lam.ho@dal.ca} 
%\And University of Science, Vietnam \\ \texttt{ngtbinh@hcmus.edu.vn}
%\And University of Delaware \\ \texttt{vucdinh@udel.edu}
%} 
%\aistatsauthor{Duy Nguyen}
%\aistatsaddress{ University of Wisconsin-Madison  \\ \texttt{dnguyen@stat.wisc.edu }
%}

\begin{abstract}
In this paper, we study the learning rate of generalized Bayes estimators in a general setting where the hypothesis class can be uncountable and have an irregular shape, the loss function can have heavy tails, and the optimal hypothesis may not be unique.
We prove that under the multi-scale Bernstein's condition, the generalized posterior distribution concentrates around the set of optimal hypotheses and the generalized Bayes estimator can achieve fast learning rate.
Our results are applied to show that the standard Bayesian linear regression is robust to heavy-tailed distributions. 
\end{abstract}

\begin{keyword}
Bayesian learning \sep posterior concentration \sep fast rate \sep heavy-tailed loss \sep Bernstein condition
%\texttt{elsarticle.cls}\sep \LaTeX\sep Elsevier \sep template
%\MSC[2010] 00-01\sep  99-00
\end{keyword}

\end{frontmatter}

%\linenumbers

\section{Introduction}

There has been a growing interest in posterior concentration rates of Bayesian inference over the last decade. 
Posterior concentration allows us to uncover frequentist properties of Bayesian methods and implies that most of the posterior mass will be close to the truth in the frequentist sense.
Studying such properties enables designs of appropriate priors for Bayesian inference in various contexts \cite{rousseau2011asymptotic, arbel2013bayesian, rousseau2016frequentist}. 

Similar approaches have also been proposed in statistical learning theory. 
In such settings, one considers models of predictors defined relative to some loss functions and proves frequentist convergence bounds of generalized Bayes predictors constructed with respect to a \emph{posterior randomization measure}.
The most notable work on this direction is the framework of ``safe Bayesian,'' where the formulation for generalized Bayesian posterior can be tuned by an optimal learning rate \cite{grunwald2011safe}.
Instead of choosing priors, within such a framework, one can construct more flexible estimators over a wide range of hypothesis spaces,  losses, and model misspecifications. 

From another perspective, the topic of fast learning rate in statistical learning has become a subject of growing interest in recent works. 
The pursuit of a ``fast rate'' regime has led to many conditions in learning theory under which fast rates are possible such as low noise assumption \cite{audibert2007fast,dinh2015learning},  stochastic mixability condition \cite{mehta2014stochastic}, Bernstein's condition \cite{bartlett2006empirical}, v-central condition \cite{vanerven15a}, and multi-scale Bernstein's condition \cite{dinh2016fast}.
Traditionally, most works in this direction have primarily focused on bounded losses, and deviations from this expected behavior are worrisome, especially when the loss of the learning problem of interest is unbounded and/or has heavy tails. 
%In many applications, the assumptions of boundedness or sub-gaussian of the losses need to be imposed, and this severely limits the results.

Recently, it has been shown that it is possible to generalize conditions for fast learning rates with unbounded and heavy-tailed losses. 
The fast learning rate for sub-gaussian and sub-exponential losses are done in the context of density estimation \cite{ zhang2006e, zhang2006information} and for general losses \cite{lecue2012general}, of which proofs of fast rates heavily employ the Bernstein's condition and the central condition. 
In \cite{hsu2016loss}, the authors provide an exponential concentration of the median-of-means estimator under heavy-tailed distributions to approximate minimization of smooth and strongly convex losses. 
Similarly, the paper \cite{mendelson2015aggregation} proposes studying the ``optimistic rate" under the small-ball condition for learning with heavy-tailed convex losses. 
Another effort to resolve this issue was shown in \cite{dinh2016fast} with their newly proposed multi-scale Bernstein's condition, which enables learning with heavy tails when the loss function is non-convex and the optimal hypothesis is not unique. 
Their analyses recover fast learning rates for empirical risk minimization (ERM) estimators under bounded losses, but, more significantly, also hold for heavy-tailed losses.   

The vast majority of the recent works in obtaining fast learning rates have taken place in the frequentist approach, whereas applications to generalized Bayesian estimators are unknown. 
In \cite{Grunwald16}, the authors take a further step to show that fast learnings in the generalized Bayesian setting are, indeed, attainable. 
However, the major drawbacks are that the optimal learning rate $\beta$ must be known in advance and that the hypothesis class is finite. 
The ``safe Bayesian'' methods \cite{grunwald2011safe} provide a framework to analyze a special form of generalized Bayes estimators employing the central condition, which cannot be applied to losses with polynomial tails \cite{vanerven15a}.
As a result, the feasibility of fast learning rates for heavy-tailed distributions under Bayesian frameworks remains unknown. 

Building upon the multi-scale Bernstein's condition, we analyze fast concentration rates of generalized Bayes estimators in a general framework where the hypothesis class can be infinite/uncountable and have an irregular shape, the loss function can have heavy tails, and the optimal hypothesis may not be unique. 
Our results demonstrate that learning rates faster than $\mathcal{O}(n^{-1/2})$ can be obtained. 
Moreover, depending on the regularity of the risk function and the complexity of the hypothesis class, the learning rate can be arbitrarily close to the optimal rate $\mathcal{O}(n^{-1})$.
We apply our results to show that the standard Bayesian linear regression is robust to heavy-tailed distributions.
Specifically, Bayesian linear regression with the regular square loss can achieve fast rate learning when the errors follow t-distributions.

\paragraph{Related work}
Bayesian framework has been applied extensively to a wide variety of research areas including ecology \cite{bachl2019inlabru}, evolutionary biology \cite{gill2017relaxed}, epidemiology \cite{ho2018birth,ho2018direct}, and economics \cite{geweke2003bayesian}.
However, theoretical properties of Bayesian methods have not been studied extensively as its frequentist counterparts, especially for heavy-tailed losses.
In particular, several frequentist approaches have been shown to perform well with heavy-tailed losses including ERM \cite{dinh2016fast,brownlees2015empirical}, median-of-means estimator \cite{hsu2016loss,lugosi2019mean}, k-mean clustering \cite{dinh2016fast,bachem2017uniform}, support vector machines \cite{christmann2009consistency}, Least Squares Estimator \cite{han2019convergence}.
Recently, much effort have been devoted to study the asymptotic theory of Bayesian inference \cite{rousseau2011asymptotic,grunwald2011safe,nguyen2016borrowing,dinh2017convergence,grunwald2017inconsistency,de2019safe}.
However, the lack of results for heavy-tailed losses has hindered the applicability of the Bayesian inference to such a scenario.
This is a major disadvantage compared to other frequentist methods.
Therefore, it is essential to establish a theoretical guarantee for Bayesian methods with heavy-tailed losses.
In this paper, we will bridge this gap for the Bayesian framework.

\section{Mathematical framework}
\label{sec:math}

Let $(\mathcal{X}, \zeta)$ be a measurable space and $Z=(X,Y)$ be a random variable taking values in $\mc{Z} = \mathcal{X} \times \mathcal{Y}$ with a probability distribution $P$ where $\mathcal{Y} \subset \mathbb{R}$. 
We assume that the hypothesis class $\mathcal{H}$ is a bounded subset of the space of square-integrable functions $L_2(\mathcal{X}, \zeta)$ with the convex hull $\overline{\mathcal{H}}$.

For a prior distribution $\mu$ on $\mc{H}$ and a set $D$ of $n$ independent and identically distributed samples $\{Z_1, Z_2, \ldots, Z_n\}$ of $Z$, the \emph{posterior randomization measure} given the data over the hypothesis space $\mathcal{H}$ has a density
\begin{align*}
p_D(h) &\propto \prod_{i=1}^n Q(Z_i \mid h) 
= \exp \left \{ - \sum_{i=1}^n \mathcal \ell(Z_i, h) \right \}
\end{align*}
with respect to $\mu$, where $\prod_{i=1}^n Q(Z_i \mid h)$ is called generalized likelihood function and $\ell: \mathcal{Z} \times \overline{\mathcal{H}} \to \mathbb{R}$ is a function defined by 
$
\ell (Z, h) = - \log Q(Z \mid h), 
$
hereafter referred to as the \emph{loss function}.

In the standard Bayesian setting, $Q(Z \mid h) = P(Z \mid h)$ where $P(Z \mid h)$ is the regular density function and the posterior randomization measure is just the posterior distribution. 
When $Q(Z \mid h) \ne P(Z \mid h)$, this setting becomes the quasi-Bayesian approach.
For various problems of Bayesian learning using \emph{mean-field variational inference}, the generalized likelihood function belongs to a family of functions that can reasonably approximate the likelihood function. 
In the ``safe Bayesian'' framework \cite{grunwald2011safe}, $Q(Z \mid h) = [P(Z \mid h)]^{\eta}$ where $\eta$ is a tuned parameter obtained by minimizing a cumulative log-loss.
It is worth noticing that there is a connection between Bayesian setting and PAC-Bayes which has been discussed elsewhere \cite[see e.g. ][and the references therein]{germain2016pac}. 

For a given set of samples $D$, the \emph{generalized Bayes estimator} is defined as 
\[ 
\hat{h} = \int_{\mc{H}} p_D(h) h d\mu.
\]
Predictions with generalized Bayes estimator are obtained by taking the average of the prediction of the hypotheses in $h$.
The estimator, thus, does not necessarily belong to $\mathcal{H}$ and is an improper estimator. 
This type of estimator has appeared in various contexts in machine learning. 
For example, as noted in \cite{grunwald2012safe}, the safe Bayesian algorithm can be regarded as just running the standard Hedge-algorithm \cite{freund1995desicion} and then making a Cesaro-averaged prediction of the previous Hedge predictions. 
Similarly, the Weighted Average algorithm \cite{kivinen1999averaging, cuong2013generalization} makes prediction based on the weighted average predictions of all the hypotheses in the hypothesis space with the weight function
\[
w(h) = \exp\left (-{c_1 \sum_{i=1}^n{|h(X_i) - Y_i|^{c_2}}}\right)
\]
and thus fits into this framework. 

We define the \emph{risk function} as
$
R(h) = \mathbb{E}_{Z\sim P}[\ell(Z, h)]
$ 
and the set of hypotheses whose risks are less than or equal to a threshold value $\gamma$ as
$
\mc{H}_\gamma = \{ h \in \mc{H}: R(h) \leq \gamma \}.
$
For convenience, we assume that 
\begin{equation}
\inf_{h \in \mc{H}} R(h) = \inf\{\gamma: \mu(\mc{H}_\gamma) > 0 \} := \gamma^ *.
\label{en:opt-risk}
\end{equation}
Here, $\gamma^*$ can be considered as ``optimal risk".

%It is worth noting that we can always achieve (\ref{en:opt-risk}) by simply removing the set 
%$
%\{ h \in \mc{H}:  R(h) < \gamma^*\}
%$
%which has measure $0$ with respect to $\mu$, from the hypothesis class $\mc{H}$.
We note that this assumption can be relaxed because the set $\{ h \in \mc{H}:  R(h) < \gamma^*\}$ has measure $0$.
The rationale is that a single best hypothesis is meaningless in the Bayesian setting when $\mc{H}$ is uncountable. 
Hence, we should compare the generalized Bayes estimator to a set of good hypotheses that has a positive measure, as suggested in \cite{freund2004generalization, cuong2013generalization}. 
The measure of such a set of ``good hypotheses'' plays a central role in our analyses and directly influences the concentration rates.

In this paper, we are interested in the concentration of the posterior around the set of optimal hypotheses $\mc{H}_{\gamma^*}$ and the convergence properties of the generalized Bayes estimator $\hat{h}$. 
Our mathematical framework is designed to analyze the problem of Bayesian learning for unbounded and/or heavy tail losses.
We recall that a random variable $S$ is said to have a heavy right tail distribution if
\[
\lim_{s \to \infty} e^{\lambda s}\mathbb{P}[S>s] = \infty
\]
for all $\lambda>0$ and the definition is similar for a heavy left tail distribution.
Learning with a heavy-tailed loss means that $\ell(Z,h)$ has a heavy tail distribution from some or all hypotheses $h \in \mathcal{H}$.
To enable analyses of fast concentration rates, we impose the following regularity conditions:

\begin{Assumption}[Regularity condition for risk function]
The risk function $R$ is convex and Lipschitz on $\overline{\mathcal{H}}$.
\label{assump:covLip}
\end{Assumption}

We observe that although the risk function $R$ is convex on the convex hull $\overline{\mathcal{H}}$, it may still have multiple global minimizers on $\mathcal{H}$ because we do not put any additional assumption on the geometry of $\mathcal{H}$.  
Figure \ref{fig:example} gives an example where this scenario happens.
In this example, the convex function $f(x,y) = x^2 + y^2$ achieves the global minimum at two different points $(-1,0)$ and $(1,0)$.

\begin{figure*}[t]
\begin{center}
\includegraphics[width=0.5\linewidth]{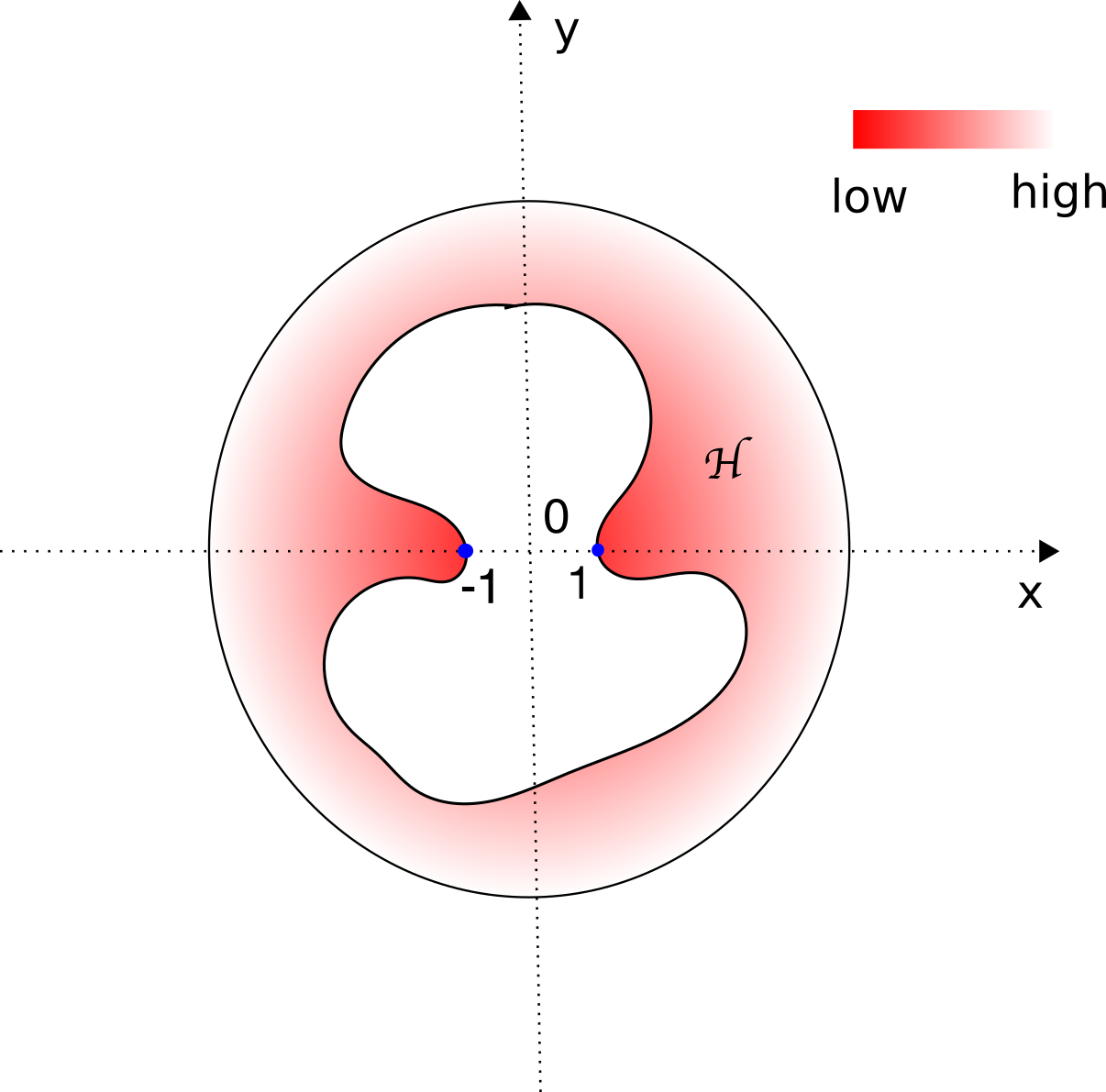}
\caption{An example where a convex function $f(x,y) = x^2 + y^2$ has two global minimizers (blue dots) on a non-convex set $\mc{H}$. 
The heat map represents the value of $f(x,y)$ on $\mc{H}$.}
\label{fig:example}
\end{center}
\end{figure*}

\begin{Assumption}[Multi-scale Bernstein's condition]
There exist a finite partition of $\mathcal{H}=\cup_{i \in I}{\mathcal{H}_i}$, positive constants $B = \{ B_i \}_{i \in I}$, constants $\alpha = \{ \alpha_i \}_{i \in I}$ in $(0, 1]$, and a finite set $\mc{H}^* = \{ h^*_i \}_{i \in I} \subset \mathcal{H}_{\gamma^*}$ such that 
\[
\mbb{E} [\ell(Z, h)-\ell(Z, h^*_i)]^2 \le B_i [R(h) - \gamma^*]^{\alpha_i}
\]
 for all $i \in I$ and $h \in \mathcal{H}_i$.
\label{assump:Bern}
\end{Assumption}

The multi-scale Bernstein's condition is a generalization of the classical Bernstein's condition introduced in \cite{dinh2016fast} to analyze fast convergence rates of the empirical risk minimizer estimator in unbounded losses settings. 
If a loss function satisfies the Bernstein's condition, then it also satisfies the multi-scale Bernstein's condition.
However, while the Bernstein's condition forces the risk function to have a unique minimizer \cite{vanerven15a},  the multi-scale Bernstein's condition does not have this restriction.
Note that this condition implies that $\mathcal{H}_{\gamma^*}$ is not empty.

\begin{Assumption}[Regularity condition for prior]
There exist $C_1(\mu) > 0$, $C_2(\mu) > 0$, and $\kappa > 0$ such that  
\[
\mu(B(h^*,\epsilon)) \geq C_1 \exp( - C_2 \epsilon^{-\kappa})
\] 
for all $h^* \in \mc{H}_{\gamma^*}$. 
Here, $B(h^*,\epsilon)$ is the ball in $L_2(\mathcal{X}, \zeta)$ with the center $h^*$ and the radius $\epsilon$.
\label{assump:mu}
\end{Assumption}

Assumption $\ref{assump:mu}$ belongs to a class of regularity assumption called \emph{prior mass assumption} and requires that the prior measures put a sufficient amount of mass near $\mathcal{H}_{\gamma^*}$. 
Such an assumption is standard in the analysis of the convergence rate of posterior measures and can be verified for a broad class of probability distributions \cite{ghosal2000convergence}. 
For example, this condition holds for the uniform distribution and the truncated normal distribution on any compact finite-dimensional manifold.
%One way to interpret this assumption is that if $\mu$ gives higher weights to hypotheses with small expected losses, the concentration rate will be faster. 
%In this work, for simplicity, we assume that the prior mass decays as a polynomial of the volume of the neighborhood, but this condition can be further relaxed. 

We also need to impose some conditions on the complexity of the hypothesis space. 
For convenience, let $\mc{G}$ denote the set of all functions $g: \mc{Z} \to \mbb{R}$ such that $g(Z)=\ell(Z,h)$ for some  $h \in \mathcal{H}$.
For $\epsilon>0$, let $\mathcal{N}(\epsilon, \mc{G}, L_2(P))$ be the \emph{covering number} of $(\mc{G}, L_2(P))$; that is, $\mathcal{N}(\epsilon, \mc{,G}, L_2(P))$ is the minimal number of balls of radius $\epsilon$ needed to cover $\mc{G}$. 
We define the \emph{universal metric entropy} of $\mc{G}$ by 
\[
H(\epsilon, \mc{G}) = \sup_{Q} \log \mathcal{N}(\epsilon, \mc{G}, L_2(Q)),
\] 
where the supremum is taken over the set of all probability measures $Q$ concentrated on some finite subset of $\mc{Z}$. 
We make the following two assumptions regarding the complexity of $\mc{G}$. 

 \begin{Assumption}[Finite covering number]
There exist $\mc{C}_1 \ge 1$ and $K_1 \ge 1$ such that
\[
\log \mathcal{N}(\epsilon, \mathcal{G}, L_2(P)) \le \mc{C}_1 \log(K_1/\epsilon) \quad \forall \epsilon \in (0,K_1].
\]
\label{assump:cov_num}
\end{Assumption}
 
 \begin{Assumption}[Universal entropy bounds]
There exist $\mc{C}_2 \ge 1$ and $K_2 \ge 1$ such that
\[
H(\epsilon, \mc{G}) \le \mc{C}_2 \log(K_2/\epsilon) \quad \forall \epsilon \in (0,K_2].
\]
\label{assump:uni_entropy}
\end{Assumption}
Denote $\mc{C} = \max \{ \mc{C}_1, \mc{C}_2 \}$. 
From now on, we will use $\mc{C}$ as the common constant for both Assumptions \ref{assump:cov_num} and \ref{assump:uni_entropy}.

Finally, we need a way to control the heavy tails of the loss functions.
We employ the integrability condition of the envelope function, which has been studied previously in \cite{dinh2016fast,lecue2012general}.

\begin{Assumption}[Integrability of the envelope function]
There exist $W>0$ and $r \ge 4 \mc{C}$ such that 
\[
\left(\mathbb{E} \sup_{g \in \mathcal{G}} {|g|^{r}}\right)^{1/r} ~\le~ W.
\]
\label{assump:int}
\end{Assumption}

For convenience, we denote the losses $\ell(Z, h)$ by $\ell(h)$, and define the empirical loss:
\[
 \ell_D( h) = \frac{1}{n}\sum_{i=1}^n{\ell(Z_i, h)}.
\]
For each hypothesis $h_0\in \mc{H}$, we define a ball $\mc{B}(h_0,\epsilon)$ of radius $\epsilon>0$ such that
\[
\mc{B}(h_0,\epsilon) = \left\{ {h\in \mc{H}: \left \{ \mbb{E} [\ell(h)-\ell(h_0)]^2 \right \}^{1/2} \leq \epsilon} \right\}.
\]

It is worth noticing that if a hypothesis belongs to a ball $\mc{B}(h^*,\epsilon)$ for some $h^* \in \mc{H}_{\gamma^*}$, then its risk is bounded by $\gamma^* + \epsilon$.
To be specific, 
\[
\bigcup_{h^* \in \mc{H}_{\gamma^*}}{ \mc{B}(h^*, \epsilon) } \subset H_{\gamma^* + \epsilon}.
\]

%%%%%%%%%%%%%%%%%%%%%%%%%%%%%%%%%%
%%%%%%%%%%%%%%%%%%%%%%%%%%%%%%%%%%

\section{Fast concentration rates}
\label{sec:fast_rate}

In this section, we prove that generalized Bayes estimators achieve fast rate learning under the Assumptions introduced in the previous section.
Our proof contains two main steps:

\paragraph{Step 1:} We prove that the posterior distribution concentrates around the set of optimal hypotheses $\mc{H}_{\gamma^*}$ exponentially fast as $n$ goes to infinity.
In other words, we prove that the posterior distribution of the hypotheses which are far away from $\mc{H}_{\gamma^*}$ converges to 0 exponentially fast.
The main technique of this step is proving that the difference in the empirical loss between a hypothesis which is close to $\mc{H}_{\gamma^*}$  and a hypothesis which is far away from $\mc{H}_{\gamma^*}$  is sufficiently large.

\paragraph{Step 2:} We bound the convergence rate of the generalized Bayes estimator to the optimal risk $\gamma^*$.
To do so, we show that the generalized Bayes estimator is very close to the average of all hypotheses near $\mc{H}_{\gamma^*}$. 
This is due to the fact that the posterior distribution of the hypotheses far away from $\mc{H}_{\gamma^*}$ is small, which was proved in Step 1. 
Therefore, the risk of the generalized Bayes estimator is close to the average risk of all hypotheses near $\mc{H}_{\gamma^*}$, which is also close to $\gamma^*$.

In the rest of the paper, for some $\beta > 0$, let $\epsilon = n^{-\beta}$ and $\mc{H}^\epsilon$ denote the finite set containing $\mc{H}^*$ such that
\[
\bigcup_{h \in \mc{H}^\epsilon} \mc{B}(h,\epsilon) = \mc{H}
\]
and that 
$
|\mathcal{H}^{\epsilon}| \le \left( K/\epsilon \right)^{\mathcal{C}} + |\mc{H}^*|,
$ 
where $\mc{H}^*$ is defined in Assumption \ref{assump:Bern}.
Note that Assumption \ref{assump:cov_num} guarantees the existence of $\mc{H}^{\epsilon}$.
We now provide the details for the proof of these two steps. 

\subsection{Posterior concentration}

\begin{Theorem}[Posterior concentration]
Assume that Assumptions \ref{assump:covLip} -- \ref{assump:int} hold.
Let $\beta$ be a positive number such that
\[
\beta <  \max \left \{ \frac{1 - 2 \sqrt{\mc{C}/r}}{2 - \min_{i \in I} \alpha_i}, \frac{1}{1 + \kappa} \right \}.
\]
Then, for any $\delta \in (0,1)$, there exist $C_{r, \beta}, C'_{r, \beta}  > 0$ and $N_{\delta, r, B, \alpha, \kappa}>0$ such that for $n \ge N_{\delta, r, B, \alpha, \kappa}$ and $\epsilon = n^{-\beta}$, we have:
%\[
%\sup_{h \in \mc{H} \setminus \mc{H}_{\gamma^* + C_{\delta, r, \beta}  \epsilon}} p_D(h) 
%\leq  \frac{1}{C_1} \exp \left \{- \frac{1}{2} \left [ C_{r, \beta} + \left ( \frac{C'_{r,\beta}}{\delta} \right )^{1/[2 \sqrt{\mc{C} r}]} \right ] n \epsilon \right \}
%\]
\[
\sup_{h \in \mc{H} \setminus \mc{H}_{\gamma^* + C_{\delta, r, \beta}  \epsilon}} p_D(h) 
\leq  \frac{1}{C_1} \exp \left \{- \frac{1}{2} \left [ C_{r, \beta} + \left ( \frac{C'_{r,\beta}}{\delta} \right )^{1/[2 \sqrt{\mc{C} r}]} \right ] n^{1 - \beta} \right \}
\]
with  probability at least $1 - \delta$. 
\label{thm:postcon}
\end{Theorem}
 
The detailed proof of this theorem is provided in the Appendix.
Here, we want to give some insights about the proof's arguments. 
Let us consider the simplest case when $\mathcal{H}$ is finite and the optimal hypothesis $h^*$ is unique. 
In this setting, for any other hypothesis $h \in \mathcal{H}$, by the strong law of large numbers, we have
\[
\ell_D(h) - \ell_D(h^*) \approx R(h) - R(h^*)
\]
as the sample size $n$ goes to infinity.
Informally, this implies that
$
p_D(h)/p_D(h^*) \to 0.
$
Since $\mc{H}$ is finite and
$
\sum_{h \in \mc{H}}{p_D(h)} = 1,
$
we deduce that the distribution concentrates around $\mc{H}_{\gamma^*} = \{ h^* \}$. 

In the case when $\mathcal{H}$ is infinite, comparing between two hypotheses becomes less meaningful. 
To extend the result, we need to provide a uniform bound on 
$
\ell_D(h) - \ell_D(h') 
$
for all  $h \in U_1$ and $h' \in U_2$, where $U_2$ is a neighborhood of $h^*$ and $U_1$ is a set that covers most of the outside of $U_2$.
This estimate is obtained by a combination of the following two Lemmas, of which the optimal hypothesis $h^*$ acts as an intermediary for comparisons. 
 
\begin{Lemma}
Assume that Assumptions \ref{assump:covLip}, \ref{assump:Bern}, \ref{assump:cov_num}, \ref{assump:uni_entropy}, and \ref{assump:int} hold.
For any $\beta < 1 - 2 \sqrt{\mc{C}/r}$, there exists $C_{r, \beta}, C'_{r, \beta}  > 0$ and such that  for all $n \in \mbb{N}$ and $\delta \in (0,1)$, we have:
\[
\left|{\ell_D(h)-\ell_D(h_0)}\right| \le \left [ C_{r, \beta} + \left ( \frac{C'_{r,\beta}}{\delta} \right )^{1/[2 \sqrt{\mc{C} r}]} \right ] \epsilon, \quad
\forall h_0 \in \mc{H}^{\epsilon},h \in \mc{B}({h_0,\epsilon})
\]
with probability at least $1 - \delta$.
\label{lem:boundemp}
\end{Lemma}

\begin{Lemma}
Assume that Assumptions \ref{assump:covLip}, \ref{assump:Bern}, \ref{assump:cov_num}, and \ref{assump:int} hold.
For any $a > 0$, $\delta \in (0,1)$, and a positive number $\beta$ statisfying
\[
\beta < (1 - 2 \sqrt{\mc{C}/r})/ (2 - \alpha_i) \quad \forall i \in I,
\]
there exists $N_{a, \delta, r, B, \alpha}>0$ such that for $n \ge N_{a, \delta, r, B, \alpha}$, we have
\[
\forall h \in \mc{H}^{\epsilon} \setminus \mc{H}_{\gamma^* + a \epsilon}, \exists h^* \in \mc{H}^*: \ell_D(h) - \ell_D(h^*)  > \frac{a \epsilon}{4}
\]
with probability at least $1-\delta$.
Here, $I$ and $\mc{H}^*$ are defined in Assumption \ref{assump:Bern}.
\label{lem:boundrisk}
\end{Lemma}

Lemma \ref{lem:boundemp} is a consequence of Lemma 3.5 in \cite{dinh2016fast} and Lemma \ref{lem:boundrisk} is 
Theorem 3.2 in  \cite{dinh2016fast}.
Lemma $\ref{lem:boundrisk}$ ensures that a hypothesis that has small empirical loss will also have small risk. 
This result provides an alternative to concentration bound, which may not exist. 
It is similar to the techniques of using one-sided inequalities for learning without concentration bound, established in \cite{mendelson2015learning}. 

Finally, when the optimal hypothesis is not unique, we need to utilize the multi-scale Bernstein's condition to partition the hypothesis spaces into regions where local behavior of the empirical loss function can be controlled, and combine the estimates in later steps. 
We note that the feasibility of this approach comes from the fact that the multi-scale Bernstein's condition is a local condition. 

From now on, to ease the notation, we denote
\[
C_{\delta, r, \beta} = \frac{1}{2} \left [ C_{r, \beta} + \left ( \frac{C'_{r,\beta}}{\delta} \right )^{1/[2 \sqrt{\mc{C} r}]} \right ].
\]

\subsection{Learning rates}

\begin{Theorem}[Learning rate]
Assume that Assumptions \ref{assump:covLip} -- \ref{assump:int} hold.
Let $\beta$ be a positive number satisfying 
\[
\beta <  \max \left \{ \frac{1 - 2 \sqrt{\mc{C}/r}}{2 - \min_{i \in I} \alpha_i}, \frac{1}{1 + \kappa} \right \}.
\]
Then, for any $\delta \in (0,1)$, there exists $N_{\delta, r, \beta, \mu, \kappa} > 0$ such that 
\[ 
\mbb{P} \left( \hat{h}_n \in \mc{H}_{\gamma^* + 2 C_{\delta,r,\beta} \epsilon}  \right) \geq 1 - \delta,
\]
for all $n \geq N_{\delta, r, \beta, \mu, \kappa}$ and $\epsilon = n^{-\beta}$.
\label{thm02}
\end{Theorem}

\begin{proof}
We define
\[
M = \sup_{h \in \mc{H}} \|h\|_2 < \infty, \quad \text{and} \quad
\nu = \int_{\mc{H}_{\gamma^* + C_{\delta,r,\beta} \epsilon}} {p_D(h) d\mu} \leq 1.
\]
Note that $M$ is finite because $\mc{H}$ is a bounded subset of $L^2(\mathcal{X}, \zeta)$.
On the other hand, by Theorem \ref{thm:postcon}, with probability at least $1 - \delta$:
\[
1 - \nu =  \int_{\mc{H} \setminus\mc{H}_{\gamma^*+C_{\delta,r,\beta} \epsilon}} {p_D(h) d\mu} \leq \frac{\exp \left \{- C_{\delta, r, \beta} n \epsilon \right \}}{C_1}.
\]
Hence, when $n$ is sufficient large, we have $\nu > 0$ with probability at least $1 - \delta$.

By Assumption \ref{assump:covLip}, $R$ is convex and Lipchitz in $\overline{\mathcal{H}}$.
Therefore, 
\[
\int_{\mc{H}} {h p_D(h) d\mu} \quad \text{and} \quad \int_{\mc{H}_{\gamma^*+C_{\delta,r,\beta} \epsilon}} {h \frac{p_D(h)}{\nu} d\mu}
\]
belong to $\overline{\mathcal{H}}$ and there exists a Lipchitz constant $L$ such that

\begin{multline*}
\left | R \left( {\int_{\mc{H}} {h p_D(h) d\mu}} \right) - R \left({\int_{\mc{H}_{\gamma^*+C_{\delta,r,\beta} \epsilon}} {h \frac{p_D(h)}{\nu} d\mu}} \right) \right | \\
\leq L \left \| \int_{\mc{H}} {h p_D(h) d\mu} - \int_{\mc{H}_{\gamma^*+C_{\delta,r,\beta} \epsilon}} {h \frac{p_D(h)}{\nu} d\mu} \right \|_2.
\end{multline*}

We deduce that
\begin{multline*}
R(\hat{h}_n) = R \left( {\int_{\mc{H}} {h p_D(h) d\mu}} \right)
 \leq R \left({\int_{\mc{H}_{\gamma^*+C_{\delta,r,\beta} \epsilon}} {h \frac{p_D(h)}{\nu} d\mu}} \right) \\
 +  L \left \| { {\int_{\mc{H} \setminus\mc{H}_{\gamma^*+C_{\delta,r,\beta} \epsilon}} {h p_D(h) d\mu}} }\right \|_2 
 + \frac{1 - \nu}{\nu} L \left \| {\int_{\mc{H}_{\gamma^*+C_{\delta,r,\beta} \epsilon}} { h p_D(h) d\mu}}  \right \|_2.
\end{multline*}

We have
\begin{multline*}
 R \left({\int_{\mc{H}_{\gamma^* + C_{\delta,r,\beta} \epsilon}} {h \frac{p_D(h)}{\nu} d\mu}} \right)
\leq \int_{\mc{H}_{\gamma^* + C_{\delta,r,\beta} \epsilon}} {R(h) \frac{p_D(h)}{\nu} d\mu} \\
 \leq \int_{\mc{H}_{\gamma^* + C_{\delta,r,\beta} \epsilon}} {(\gamma^* + C_{\delta,r,\beta} \epsilon) \frac{p_D(h)}{\nu} d\mu} = \gamma^* + C_{\delta,r,\beta} \epsilon.
\end{multline*}
Moreover,
\[
\left \| { {\int_{\mc{H} \setminus\mc{H}_{\gamma^*+C_{\delta,r,\beta} \epsilon}} {h p_D(h) d\mu}} }\right \|_2 
\leq {\int_{\mc{H} \setminus\mc{H}_{\gamma^*+C_{\delta,r,\beta} \epsilon}} {\| h \|_2 p_D(h) d\mu}}  
\leq  M (1- \nu).\,
\]
and
\[
\left \| {\int_{\mc{H}_{\gamma^*+C_{\delta,r,\beta} \epsilon}} { h p_D(h) d\mu}}  \right \|_2 
\leq {\int_{\mc{H}_{\gamma^*+C_{\delta,r,\beta} \epsilon}} {\| h \|_2 p_D(h) d\mu}}  
\leq M \nu.
\]

We conclude that with probability at least $1 - \delta$,
\[
R(\hat{h}_n)  \leq \gamma^* + C_{\delta,r,\beta} \epsilon + 2 L M (1 - \nu)
 \leq \gamma^* + C_{\delta,r,\beta} \epsilon + 2 L M \frac{\exp \left \{- C_{\delta, r, \beta} n \epsilon \right \}}{C_1}.
\]

Hence, when $n$ is sufficiently large, we have
\[
R(\hat{h}) \leq  \gamma^* + 2 C_{\delta,r,\beta} \epsilon
\]
with probability at least $1 - \delta$, which completes the proof for the theorem.
\end{proof}
 
The result of Theorem $\ref{thm02}$ implies 
\begin{Corollary}
For all $\delta \in (0, 1)$, 
$
R(\hat h_n) = \gamma^* +  \mathcal{O}(n^{-\beta})
$
with probability at least $1-\delta$, where 
\[
\beta <  \max \left \{ \frac{1 - 2 \sqrt{\mc{C}/r}}{2 - \min_{i \in I} \alpha_i}, \frac{1}{1 + \kappa} \right \}.
\]
When $r$ is sufficiently large, $\min_{i \in I} \alpha_i =1$, and $\kappa$ is sufficiently small, we achieve convergence rates arbitrarily close to $\mathcal{O}(n^{-1})$. 
\end{Corollary}

Hence, fast learning rates for generalized Bayesian estimators are available within our framework. 
In general, the order of convergence depends on the regularity of the loss function (via the multi-scale Bernstein's order) and the balance between the complexity of the hypothesis class and the thickness of the tail of the loss's distribution.

\section{Robustness of Bayesian linear regression}
\label{sec:ex}

In this section, we will apply our results to show that Bayesian linear regression is robust to heavy-tailed distributions.
To be specific, we consider the following standard linear regression setting:
\[
Y_i = \mb{X}_i u_0 + \epsilon_i, ~~i=1,2,\ldots,n
\]
where $Y_i \in \mbb{R}$, $\mb{X}_i \in \mc{X} \in  \mbb{R}^{d}$, $u_0 \in \mbb{R}^d$, and $\epsilon_i$ are i.i.d random variables which follow a t-distribution with degree of freedom $k$.

We will prove that even if we do not know that $\epsilon_i$ is heavy-tailed and just assume that $\epsilon_i$ follows a standard normal distribution, the Bayesian linear regression still achieves fast rate learning.
Given a proper prior $\pi_u$ for $u$, the posterior distribution of $u$ has the following form:
\begin{equation}
p_D(u \mid \{Y_i\}_{i=1}^n) \propto \pi_u \exp \left \{ \sum_{i=1}^n {- \frac{(Y_i - \mb{X}_i u)^2}{2}} \right \},
\label{eqn:Bayes}
\end{equation}
which corresponds to our setting with the loss function $\ell(Y, \mb{X}, u) = ( Y - \mb{X} u )^2$.

\begin{Theorem}
Assume that $\| u_0 \|_2 \leq M_u$, $\mc{X}$ is bounded in $\|.\|_2$ by $M_X$, $k > 4 d$, and $\pi_u$ is regular (Assumption \ref{assump:mu}).
Let $\beta$ be a positive number satisfying 
\[
\beta < \min \{1 - 2 \sqrt{d/k}, 1/(1 + \kappa) \}.
\]
The Bayesian linear regression estimator
\[
\hat{u} = \int {u \cdot p_D(u \mid \{Y_i\}_{i=1}^n) du}
\]
achieves learning rate $n^{-\beta}$.
\label{thm:ex}
\end{Theorem}

The proof of this theorem is in the Appendix.
We observe that when $k > 16 d$, Theorem \ref{thm:ex} implies the Bayesian linear regression estimator achieves fast learning rate.
It is worth noticing that most of the common priors on a bounded set of $\mbb{R}^d$ (for example, uniform distribution) satisfy the regularity condition for prior (Assumption \ref{assump:mu}) with any $\kappa >0$.

\paragraph{Simulations}
To illustrate the result, we use the \texttt{R}-platform to simulate data from the following model:
\[
Y_i = 1 + X^{(1)}_i  + X^{(2)}_i + \epsilon_i, ~~i=1,2,\ldots,n
\]
where $\{X^{(1)}_i\} , \{X^{(2)}_i\}$ are i.i.d. random variables that follow a truncated standard normal distribution (the truncation value is $1$), and $\epsilon_i$ are i.i.d random variables which follow a t-distribution with degree of freedom $k = 5, 10, 20$.
For each degree of freedom, we vary the sample size from $10$ to $10240$ ($n = 10, 20, 40, 80, 160, 320, 640, 1280, 2560, 5120, 10240$), and for each sample size, we simulate $100$ data sets.
We analyze each data set using the standard linear regression (ERM) and the Bayesian linear regression \eqref{eqn:Bayes} with a uniform prior on the ball which is centered at $0$ and has radius $10$.
We explore the generalized posterior distribution of the coefficients using the Metropolis algorithm implemented in the \texttt{R} function \texttt{MCMCmetrop1R} from the package \texttt{MCMCpack} \cite{martin2011}.
We discard the first $20000$ iterations of the Markov chain Monte Carlo and use the next $100000$ iterations to approximate the Bayesian linear regression estimator.
Then, we apply Monte Carlo method to approximate the risk of the estimators and fit a linear regression between the risk (in log-scale) and the sample size (in log-scale) to approximate the rate of convergence of the ERM and the Bayesian linear regression (e.g. Figure \ref{fig:sim}).
We summarize the result of our simulations in Table \ref{tab:res}.
The result confirms that Bayesian linear regression is robust to heavy-tailed noises.
We note that the empirical convergence rate of the Bayesian linear regression (as well as the ERM, which has been investigated in \cite{dinh2016fast}) is faster than its theoretical bound in Theorem \ref{thm:ex}.

\begin{figure}[t]
\begin{center}
\includegraphics[width=1\linewidth]{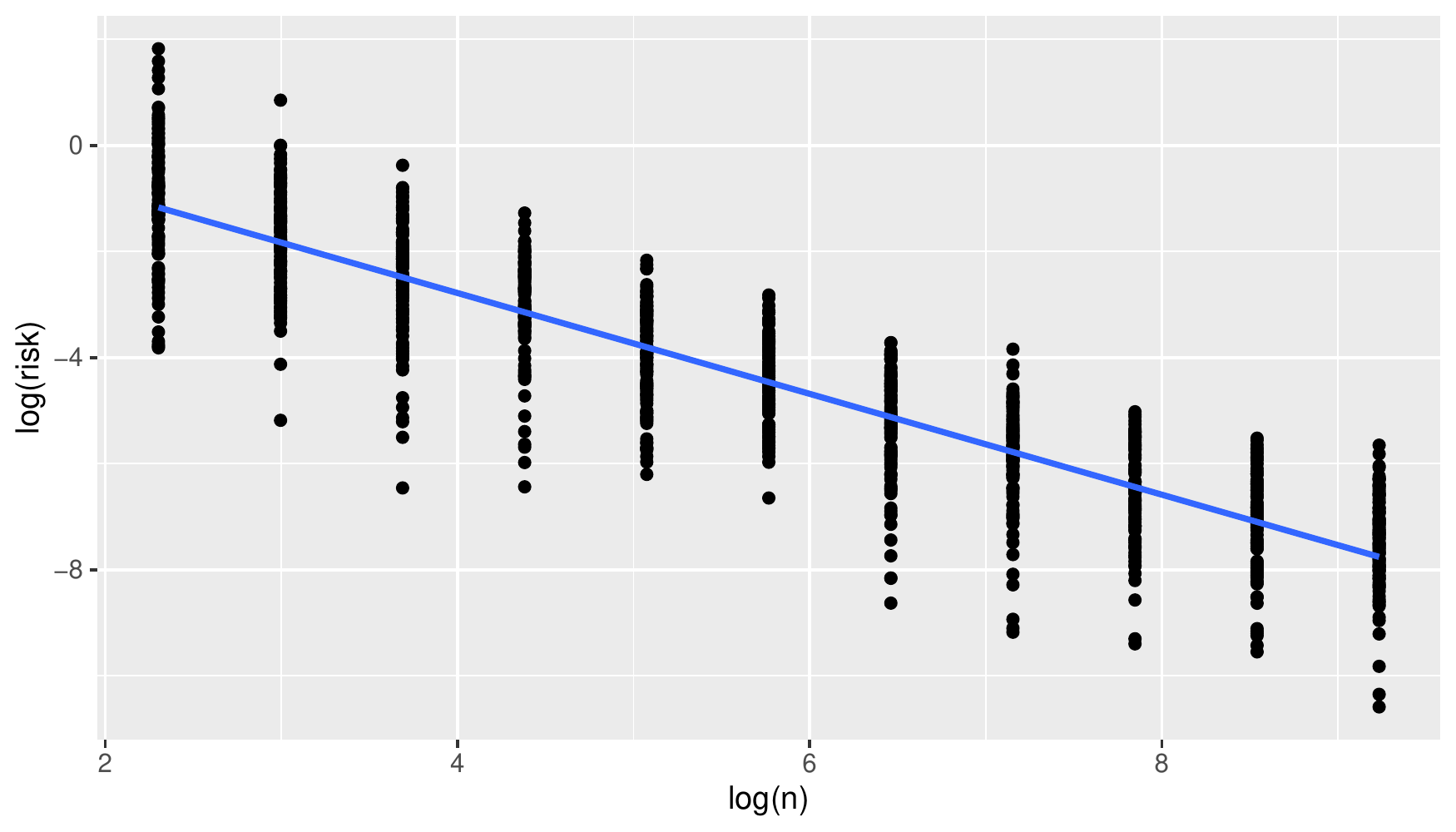}
\caption{A linear regression between the risk (in log-scale) of the Bayesian linear regression and the sample size (in log-scale) when the errors follow the $t$-distribution with $5$ degrees of freedom. The slope of the fitted line approximates the rate of convergence of the estimator.}
\label{fig:sim}
\end{center}
\end{figure}

\begin{table}[h]
\begin{center}
\begin{tabular}{c||c|c}
 Degree of freedom & ERM & Bayesian \\
 \hline
 5 & -0.984 & -0.951 \\
 10 & -1.001 & -0.966 \\
 20 & -1.046 & -0.996
\end{tabular}
\caption{The approximated rate of convergence of the two estimators with $k = 5,10,20$. }
\label{tab:res}
\end{center}
\end{table}

\section{Discussions and Conclusions}
 \label{sec:discussion}

The result of this paper indicates that learning with Bayesian estimators and heavy-tailed losses can obtain convergence rates up to an essential order 
\[
\mc{O}\left (n^{-(1 - 2 \sqrt{\mc{C}/r})/(2 - \min \{ \alpha\})}\right)
\]
where $\alpha$ is the multi-scale Bernstein's order and $r $ is the degree of integrability of the loss. 
This result is consistent with previous works using a frequentist approach \cite{dinh2016fast}. 
We note that for bounded and strongly convex losses, our assumptions can be validated with $\alpha=1$,  $I=1$, and $r= \infty$ and this reduces to the convergence rate $\mc{O}(1/n)$. 

There are several avenues for improvement. 
Firstly, in this work, we consider a setting where the generalized likelihood function has the form $\prod_{i=1}^n Q(Z_i \mid h)$.
%In other popular approaches in Bayesian statistics, for example in variational Bayes and approximate Bayesian computation, this assumption may not hold. 
In some scenarios, for example, when data are dependent, this setting may not hold.
It would be interesting to see if concentration and learning rates retain in those cases. 
Secondly, although our framework (which relies on the multi-scale Bernstein's condition) allows us to analyze the convergence of generalized Bayes estimators in more general settings than previous approaches, our result requires high-order moments of the loss to guarantee convergence. 
Recently, there has been a growing interest in fast learning rate for convex losses using the small-ball condition \cite{mendelson2015aggregation}, which requires only low-order moments. 
We would like to extend the result in this paper to study and adapt this condition to the case when the optimal hypothesis is non-unique.

Finally, the simulations confirm the robustness of Bayesian linear regression to heavy-tailed noises.
This is an assurance for end-users that Bayesian linear regression is not vulnerable to the violation of the assumption of normal errors.
In particular, no special treatment is needed when the errors follow a t-distribution and the Bayesian estimates converge to the true values at the same rate as their frequentist counterparts.
It is worth noticing that the simulations indicate that the convergence rate is $\mc{O}(1/n)$, which means that our theoretical upper bounds may not be optimal.
An interesting direction for future research is to derive sharper upper bounds and/or lower bounds.

\section*{Acknowledgments}

LSTH was supported by startup funds from Dalhousie University, the Canada Research Chairs program, the NSERC Discovery Grant RGPIN-2018-05447, and the NSERC Discovery Launch Supplement DGECR-2018-00181.

\clearpage

\bibliography{bib}

\newpage

\appendix

\section{Detailed proofs}

\begin{proof}[Proof of Theorem $\ref{thm:postcon}$]

We denote 
\[
r_n(h) =  \exp \{ - \ell_D(h) \}.
\] 
Then, the posterior can be calculated by the following formula:
\[
p_D(h) = \left[{\frac{r_n(h)}{\left\|{r_n(h)}\right\|_n}}\right]^{n},
\]
where
\[
\left\|{r_n(h)}\right\|_n = \left({\int_\mc{H} {|r_n(h)|^n d\mu}}\right)^{1/n}.
\]

For any $i \in I$, we apply Lemma \ref{lem:boundrisk} with $a = 12 C_{\delta, r, \beta}$ to obtain
\[
\ell_D(h) - \ell_D(h_i^*) > 3 C_{\delta, r, \beta} \epsilon, \quad
\forall h \in (\mc{H}_i \setminus \mc{H}_{\gamma^* + 12 C_{\delta, r, \beta} \epsilon}) \cap \mc{H}^{\epsilon}
\]
with probability $1 - \delta$.

By Lemma \ref{lem:boundemp}, we derive that
\[
\ell_D(h) - \ell_D(h') > C_{\delta, r, \beta} \epsilon, \quad
\forall h \in \mc{H}_i \setminus \mc{H}_{\gamma^* + 12 C_{\delta, r, \beta} \epsilon},~h' \in \mc{B}(h_i^*,\epsilon)
\]
with probability $1 - 3 \delta$.

Hence, 
\[
r_n(h) \leq e^{ - C_{\delta, r, \beta} \epsilon} r_n(h'), \quad
\forall h \in \mc{H}_i \setminus \mc{H}_{\gamma^* + 12 C_{\delta, r, \beta} \epsilon},~h' \in \mc{B}(h_i^*,\epsilon)
\]
with probability at least $1 - 3\delta$. 

Therefore,
\[
\sup_{h \in \mc{H}_i \setminus \mc{H}_{\gamma^* + 12 C_{\delta, r, \beta} \epsilon}} r_n(h) \leq e^{ - C_{\delta, r, \beta} \epsilon} \inf_{h' \in \mc{B}(h_i^*, \epsilon)} r_n(h'). 
\]
with probability at least $1 - 3\delta$. 

We have
\[
\|r_n\|_n  = \left( \int_{\mc{H}}{|r_n(h)|^n d\mu} \right)^{1/n} 
 \geq \left( \int_{ \mc{B}(h_i^*, \epsilon)}{|r_n(h)|^n d\mu} \right)^{1/n} 
 = \inf_{h' \in \mc{B}(h_i^*, \epsilon)} r_n(h') \mu ( \mc{B}(h_i^*,  \epsilon))^{1/n},
\]
with probability at least $1 - 3\delta$. 

Consequently, when $n$ is sufficient large,
\begin{align*}
\sup_{h \in \mc{H}_i \setminus \mc{H}_{\gamma^* + 12 C_{\delta, r, \beta} \epsilon}} p_D(h) 
&= \sup_{h \in \mc{H}_i \setminus \mc{H}_{\gamma^* + 12 C_{\delta, r, \beta} \epsilon}} \left ( \frac{r_n(h)}{\|r_n\|_n} \right )^n 
\leq \frac{e^{- 2C_{\delta, r, \beta} n \epsilon}}{\mu ( \mc{B}(h_i^*, \epsilon))} \\
& \leq \frac{1}{C_1} \exp( - n 2 C_{\delta, r, \beta} \epsilon + C_2\epsilon^{- \kappa}) \leq \frac{1}{C_1} \exp( - n C_{\delta, r, \beta} \epsilon),
\end{align*}
with probability at least $1 - 3\delta$.

Under Assumption \ref{assump:Bern}, $I$ is finite and $\mc{H} = \bigcup_{i \in I} \mc{H}_i$. 
Therefore, the proof is completed by taking a union bound over $I$.
\end{proof}

\begin{proof}[Proof of Theorem \ref{thm:ex}]
Let $u_0$ be the true value of $u$.
We will verify Assumptions \ref{assump:covLip}, \ref{assump:Bern}, \ref{assump:cov_num}, and \ref{assump:int}.
Instead of checking Assumption \ref{assump:uni_entropy}, we will prove Lemma \ref{lem:boundemp} directly.

Assumption \ref{assump:covLip}: 
The risk function $R(u) = \mbb{E}[ (Y - \mb{X} u )^2]$ is convex and Lipschitz in $u$.
Indeed,
\begin{align*}
R \left ( \frac{u_1 + u_2}{2} \right ) &= \mbb{E} \left [  \left (Y - \mb{X} \frac{u_1 + u_2}{2} \right )^2 \right ] \\ 
& \leq \frac{1}{2} ( \mbb{E}[ (Y - \mb{X} u_1 )^2] + \mbb{E}[ (Y - \mb{X} u_2 )^2]) \\
& = \frac{1}{2}(R(u_1) + R(u_2)),
\end{align*}
and
\begin{multline*}
|R(u_1) - R(u_2)| = | \mbb{E}[ (Y - \mb{X} u_1 )^2 - (Y - \mb{X} u_2 )^2] |\\
\leq  M_X \|u_2 - u_1\|_2  [2 \mbb{E}|Y - \mb{X}u_0| +4 M_X M_u ]\\
\leq M_X \|u_2 - u_1\|_2  \{2 [\mbb{E}(Y - \mb{X}u_0)^2]^{1/2} +4 M_X M_u \} \\
= M_X \left (\frac{2 k^{1/2}}{(k - 2)^{1/2}} +4 M_X M_u \right ) \|u_2 - u_1\|_2.
\end{multline*}

Assumption \ref{assump:Bern}: 
We first note that $u_0$ is the only optimal hypothesis.
Indeed,
\begin{align*}
R(u) - R(u_0) &= \mbb{E}[\mb{X}(u - u_0) (2 Y - \mb{X}(u + u_0))] \\
& = \mbb{E}[ \mbb{E} [\mb{X}(u - u_0) (2 Y - \mb{X}(u + u_0)) \mid \mb{X} ] \\
& = \mbb{E}[ \mb{X} (u - u_0) (2\mbb{E} [ Y \mid \mb{X}] - \mb{X}(u + u_0)) ] \\
& = \mbb{E}([\mb{X} (u - u_0)]^2).
\end{align*}
\[
\inf_{u \ne u_0} \frac{R(u) - R(u_0)}{\| u - u_0 \|^2_2} = \inf_{\| z \|_2 = 1} \mbb{E}([\mb{X} z]^2).
\]

Since $\mbb{E}([\mb{X} z]^2) > 0$ for all $\| z \|_2 = 1$ then $\inf_{\| z \|_2 = 1} \mbb{E}([\mb{X} z]^2) \geq D > 0$.
Therefore, $u_0$ is the only optimal hypothesis and $R(u) - R(u_0) \geq D \|u - u_0\|^2$.
Note that we have proved $\mbb{E}\{[(Y - \mb{X} u_1 )^2 - (Y - \mb{X} u_2 )^2]^2\} \leq C_0 \|u_1 - u_2\|^2$.
We conclude that the multi-scale Bernstein condition is satisfied with $\alpha = 1$.

Assumption \ref{assump:cov_num}:
\begin{align*}
& [d_P(u_1, u_2)]^2 = \mbb{E}\{[(Y - \mb{X} u_1 )^2 - (Y - \mb{X} u_2 )^2]^2\} \\
& =  \mbb{E}\{[\mb{X} (u_1 - u_2)]^2[2Y - \mb{X}(u_1 + u_2)]^2\} \\
& \leq M_X \|u_1 - u_2\|_2^2 \mbb{E}\{[2Y - \mb{X}(u_1 + u_2)]^2\} \\
& \leq M_X \|u_1 - u_2\|_2^2 \{8\mbb{E}[(Y - \mb{X}u_0)^2] + 32M_X^2 M_u^2\} \\
& = M_X \left [ 8 \frac{k}{k - 2} + 32 M_X^2 M_u^2 \right ] ||u_1 - u_2||^2 = C_0 \|u_1 - u_2\|^2.
\end{align*}

Therefore, Assumption \ref{assump:cov_num} holds with $\mc{C}_1 = d$ and $K_1 = M_u$.

Assumption \ref{assump:int}:
\begin{align*}
\mbb{E}[ \sup_{u} ( Y - \mb{X} u )^{r}] & \leq \mbb{E}[ \sup_{u} (|Y - \mb{X} u_0 | + 2 M_u M_X)^{r}] \\
&=  \mbb{E}[(| Y - \mb{X} u_0 |+ 2 M_u M_X)^{r}] \leq W
\end{align*}
when $r < k$.
Then Assumption \ref{assump:int} is satisfied with any $r \in [4 d, k)$.

Lemma \ref{lem:boundemp}:
Note that
\[
\frac{[d_P(u_1, u_2)]^2}{\|u_1 - u_2\|_2^2} 
= \mbb{E} \left \{\left [\mb{X} \frac{u_1 - u_2}{\| u_1 - u_2 \|_2} \right ]^2[2Y - \mb{X}(u_1 + u_2)]^2 \right \} > 0
\]
for all $u_1, u_2$.
Since $u_1, u_2$ are bounded, we have
\[
\frac{[d_P(u_1, u_2)]^2}{\|u_1 - u_2\|_2^2} \geq D > 0.
\]

Therefore,
\begin{align*}
& |\ell_D(u_1) - \ell_D(u_2)| \\
& = \left | \frac{1}{n} \sum_{i=1}^n{(Y_i - \mb{X}_i u_1)^2 - (Y_i - \mb{X}_i u_2)^2} \right | \\
& \leq  \frac{1}{n} \sum_{i=1}^n{2 (| Y_i - \mb{X}_i u_0 | + M_X M_u) M_X \|u_1 -u_2 \|_2} \\
& \leq \frac{1}{n} \sum_{i=1}^n{2 (| Y_i - \mb{X}_i u_0 | + M_X M_u) M_X \frac{d_P(u_1, u_2)}{D}}.
\end{align*}

So,
\[
\sup_{u_1 \in \mc{H}, u_2 \in \mc{B}(u_1, \epsilon)} |\ell_D(u_1) - \ell_D(u_2)| 
\leq \frac{1}{n} \sum_{i=1}^n{2 (| Y_i - \mb{X}_i u_0 | + M_X M_u) M_X \frac{\epsilon}{D}}.
\]

Note that, for all $M > 0$,
\[
\Pr \left (\frac{1}{n} \sum_{i=1}^n {| Y_i - \mb{X}_i u_0 |} \geq M \right ) \leq \frac{\mbb{E}| Y_1 - \mb{X}_1 u_0 |}{M} 
\leq \frac{[\mbb{E}( Y_1 - \mb{X}_1 u_0)^2]^{1/2}}{M} \leq \frac{k^{1/2}}{M (k - 2)^{1/2}}.
\]

Hence, we can choose $M_\delta$ such that
\[
\sup_{u_1 \in \mc{H}, u_2 \in \mc{B}(u_1, \epsilon)} |\ell_D(u_1) - \ell_D(u_2)| 
\leq \frac{2(M_\delta + M_X M_u)M_X}{D}\epsilon
\]
with probability at least $1 - \delta$.

\end{proof}

\end{document}